\documentclass[12pt,leqno]{article}
\usepackage{amsfonts,amsthm,amsmath}

\usepackage{amssymb}

\theoremstyle{plain}
\newtheorem{thm}{Theorem}[section]

\theoremstyle{definition}
\newtheorem{df}{Definition}[section]
\newtheorem{rem}{Remark}[section]
\newtheorem{ex}{Example}[section]

\newcommand{\FF}{\mathbb{F}}
\newcommand{\ZZ}{\mathbb{Z}}

\newcommand{\NN}{\mathbb{N}}

\DeclareMathOperator{\wt}{wt}

\begin{document}

\title{{On the cycle index and the weight enumerator}
\footnote{This work was supported by JSPS KAKENHI (15K04775, 17K05164).}
}
\author{
Tsuyoshi Miezaki
\thanks{Faculty of Education, University of the Ryukyus, Okinawa  
903--0213, Japan 
miezaki@edu.u-ryukyu.ac.jp
(Corresponding author)
}
and
Manabu Oura
\thanks{Graduate School of Natural Science and Technology, 
Kanazawa University,  
Ishikawa 920--1192, Japan 
oura@se.kanazawa-u.ac.jp
telephone: +81-76-264-5635, Fax: +81-76-264-6065 
}
}

\date{}

\maketitle

\begin{abstract}
In this paper, we introduce the 
concept of the complete cycle index and 
discuss a relation with the 
complete weight enumerator 
in coding theory. 
This work was motivated by 
Cameron's lecture notes 
``Polynomial aspects of codes, matroids and permutation groups.'' 
\end{abstract}


\vspace{-0.15in}

\quad


\section{Introduction}
In \cite{CameronPaper,Cameron}, 
a relationship between 
the cycle index and the weight enumerator was given. 
To state our results, 
we review this relationship.

Let $G$ be a permutation group on a set $\Omega$, 
where $|\Omega| = n$. 
For each element
$h \in G$, we can decompose the permutation $h$ into a product 
of disjoint cycles; let
$c_i(h)$ be the number of $i$-cycles occurring in this decomposition. 
Let $\NN$ be the set of natural numbers. 
Now the cycle
index of $G$ is the polynomial $Z(G;s_i:i\in \NN)$ in indeterminates 
$\{s_i\}_{i\in \NN}$ given by
\[
Z(G;s_i:i\in \NN) =
\sum_{h\in G}
\prod_{i\in\NN}
s_i^{c_i(h)}.
\]

Let $\FF_q$ be the finite field of order $q$. 
Let $C$ be a linear $[n,k]$ code over $\FF_q$, 
namely a $k$-dimensional subspace of $\FF_q^n$. 
The weight enumerator $w_C(x,y)$ of the code $C$ is the homogeneous polynomial
\[
w_C(x,y) = 
\sum_{{\bf c}\in C}
x^{n-\wt({\bf c})} y^{\wt({\bf c})} =
\sum_{i=0}^n A_i x^{n-i}y^i,
\]
where $A_i=\sharp\{{\bf c}\in C\mid \wt( {\bf c})=i\}$. 

We construct from $C$ a permutation
group $G(C)$ whose cycle index is essentially the weight enumerator of $C$.
The group we construct is the additive group of $C$. 
We let it act on the set
$\{1,\ldots,n\}\times \FF_q$ in the following way: the codeword
$(a_1,\ldots,a_n)$ acts as the permutation
\[
(i,x) \mapsto (i;x+a_i)
\]
of the set $\{1,\ldots,n\}\times \FF_q$. Then we have the following result.

\begin{thm}[{\cite[Proposition 7.2]{Cameron}}]\label{thm:Cameron}
We have
\[
w_C(x,y) = Z(G(C); s_1 \leftarrow x^{1/q}, s_p\leftarrow y^{p/q}), 
\]
where $q$ is a power of the prime number $p$.
\end{thm}

A generalization of the weight enumerator is known as 
the complete weight enumerator of genus $g$:
\[
w_C^{(g)}(x_{{\bf a}}:{\bf a}\in \FF_q^g)=\sum_{{\bf v_1},\ldots,{\bf v_g}\in C}
\prod_{{\bf a}\in \FF_q^g}x_{{\bf a}}^{n_{{\bf a}}({\bf v_1},\ldots,{\bf v_g})}, 
\]
where $n_{{\bf a}}({\bf v_1},\ldots,{\bf v_g})$ denotes the number 
of $i$ such that ${\bf a}=(v_{1i},\ldots,v_{gi})$. 
This gives rise to a natural question: is there
a generalization of the cycle index that relates the 
complete weight enumerator $w_C^{(g)}(x_{{\bf a}}:{\bf a}\in \FF_q^g)$?
The aim of the present paper is to provide a candidate generalization that answers this. 
We now present the concept of the complete cycle index. 

\begin{df}
Let $G$ be a permutation group on a set $\Omega$, 
where $|\Omega| = n$. 
For each element
$h \in G$, we can decompose the permutation $h$ into a product 
of disjoint cycles; let
$c(h,i)$ be the number of $i$-cycles occurring by the action of $h$. 
Now the complete cycle
index of $G$ is the polynomial $Z(G;s(h,i): h\in G,i\in \NN)$ in indeterminates 
$\{s(h,i)\mid h\in G,i\in \NN\}$ given by
\[
Z(G;s(h,i): h\in G,i\in \NN) =
\sum_{h\in G}
\prod_{i\in \NN}s(h,i)^{c(h,i)}.
\]
\end{df}

\begin{rem}
Note that if we let $s_i=s({h,i})$, then we obtain the cycle index: 
\[
Z(G;s_i: h\in G,i\in \NN\})=\sum_{h\in G}\prod_{i\in \NN}s_{i}^{n(h,i)}. 
\]
\end{rem}

The main result of this paper, Theorem \ref{thm:main}, 
uses the concept of the complete cycle index. 
We also give a generalization of Theorem \ref{thm:Cameron}. 

This paper is organized as follows. 
In Section $2$, 
we give the concept of 
the higher cycle index and the complete cycle index and 
also give the main result of this paper and its proof. 
In Section $3$, 
we give a $\ZZ_k$-code analog of the main result. 

\section{Higher cycle index and complete cycle index}

\subsection{Definitions and examples}
In this section, 
we give the concept of 
the higher cycle index and the complete cycle index, 
and provide some examples. 

\begin{df}\label{df:higheraction}
Let $C$ be a linear $[n,k]$ code over $\FF_q$. 
We construct from $C^g:=\underbrace{C\times \cdots\times C}_{g}$ a permutation
group $G(C^{g})$.
The group we construct is the additive group of $C^{g}$. 
We denote an element of $C^g$ by 
\[
({\bf c_1},\ldots,{\bf c_n}):=
\begin{pmatrix}
a_{11}&\ldots&a_{1n}\\
a_{21}&\ldots&a_{2n}\\
\vdots&\ldots&\vdots\\
a_{g1}&\ldots&a_{gn}
\end{pmatrix},
\]
where 
${\bf c_i}:={}^t(a_{1i},\ldots,a_{gi})\in C$. 
We let it act on the set
$\{1,\ldots,n\}\times \FF_q^{g}$ in the following way: 
$({\bf c_1},\ldots,{\bf c_n})$
acts as the permutation
\[
\left(i,
\begin{pmatrix}
x_{1}\\
x_{2}\\
\vdots\\
x_{g}
\end{pmatrix}
\right)\mapsto \left(i,
\begin{pmatrix}
x_{1}+a_{1i}\\
x_{2}+a_{2i}\\
\vdots\\
x_{g}+a_{gi}
\end{pmatrix}
\right) 
\]
of the set $\{1,\ldots,n\}\times \FF_q^g$. We call the 
cycle index 
\[
Z(G(C^{g}),s_i:i\in \NN) 
\]
the higher cycle index of genus $g$ for code $C$. 
We call the complete 
cycle index 
\[
Z(G(C^{g}),s(h,i):h\in C^g,i\in \NN) 
\]
the complete cycle index of genus $g$ for code $C$. 
\end{df}

\begin{rem}
Note that let $s_i=s({h,i})$. 
Then we obtain the higher cycle index: 
\[
Z(G(C^g);s_i:h\in G(C^g))=\sum_{g\in G(C^g)}\prod_{i\in \NN}s_{i}^{n(h,i)}. 
\]
\end{rem}

We now give some examples. 

\begin{ex}
Let $C=\FF_2^2$．Then the higher cycle index, the 
complete cycle index, and the complete weight enumerator of genus $2$ are as follows:
\begin{align*}
\bullet\ Z&(G(C^2);s_1,s_2)=s_1^8+6s_1^4s_2^2+9s_2^4, \\
\bullet\ Z&(G(C^2);s(h,i):h\in C^2)\\
&=
s(
\begin{bmatrix}
0&0\\
0&0
\end{bmatrix},1
)^4
s(
\begin{bmatrix}
0&0\\
0&0
\end{bmatrix},1
)^4+
s(
\begin{bmatrix}
0&0\\
0&1
\end{bmatrix},1
)^4
s(
\begin{bmatrix}
0&0\\
0&1
\end{bmatrix},2
)^2\\
&+
s(
\begin{bmatrix}
0&0\\
1&0
\end{bmatrix},2
)^2
s(
\begin{bmatrix}
0&0\\
1&0
\end{bmatrix},1
)^4
+
s(
\begin{bmatrix}
0&0\\
1&1
\end{bmatrix},2
)^2
s(
\begin{bmatrix}
0&0\\
1&1
\end{bmatrix},2
)^2\\
&+
s(
\begin{bmatrix}
0&1\\
0&0
\end{bmatrix},1
)^4
s(
\begin{bmatrix}
0&1\\
0&0
\end{bmatrix},2
)^2
+
s(
\begin{bmatrix}
0&1\\
0&1
\end{bmatrix},1
)^4
s(
\begin{bmatrix}
0&1\\
0&1
\end{bmatrix},2
)^2\\
&+
s(
\begin{bmatrix}
0&1\\
1&0
\end{bmatrix},2
)^2
s(
\begin{bmatrix}
0&1\\
1&0
\end{bmatrix},2
)^2
+
s(
\begin{bmatrix}
0&1\\
1&1
\end{bmatrix},2
)^2
s(
\begin{bmatrix}
0&1\\
1&1
\end{bmatrix},2
)^2\\
&+
s(
\begin{bmatrix}
1&0\\
0&0
\end{bmatrix},2
)^2
s(
\begin{bmatrix}
1&0\\
0&0
\end{bmatrix},1
)^4
+
s(
\begin{bmatrix}
1&0\\
0&1
\end{bmatrix},2
)^2
s(
\begin{bmatrix}
1&0\\
0&1
\end{bmatrix},2
)^2\\
&+
s(
\begin{bmatrix}
1&0\\
1&0
\end{bmatrix},2
)^2
s(
\begin{bmatrix}
1&0\\
1&0
\end{bmatrix},1
)^4
+
s(
\begin{bmatrix}
1&0\\
1&1
\end{bmatrix},2
)^2
s(
\begin{bmatrix}
1&0\\
1&1
\end{bmatrix},2
)^2\\
&+
s(
\begin{bmatrix}
1&1\\
0&0
\end{bmatrix},2
)^2
s(
\begin{bmatrix}
1&1\\
0&0
\end{bmatrix},2
)^2
+
s(
\begin{bmatrix}
1&1\\
0&1
\end{bmatrix},2
)^2
s(
\begin{bmatrix}
1&1\\
0&1
\end{bmatrix},2
)^2\\
&+
s(
\begin{bmatrix}
1&1\\
1&0
\end{bmatrix},2
)^2
s(
\begin{bmatrix}
1&1\\
1&0
\end{bmatrix},2
)^2
+
s(
\begin{bmatrix}
1&1\\
1&1
\end{bmatrix},2
)^2
s(
\begin{bmatrix}
1&1\\
1&1
\end{bmatrix},2
)^2, \\
\bullet\ w&_C^{(2)}(x_{00},\ldots,x_{11})=
\sum x_{ij}^2
+
2\sum x_{ij}x_{ki}. 
\end{align*}
\end{ex}








\subsection{Main results}

In this section, we present the main result of this paper. 
The following theorem is a generalization of Theorem \ref{thm:Cameron}.
 
\begin{thm}\label{thm:main}
Let $C$ be a code over $\FF_q$ of length $n$, 
where $q$ is a prime power of $p$.
Let $w_C^{(g)}(x_{\bf a}:{\bf a}\in \FF_q^g)$ be the 
complete weight enumerator of genus $g$ and 
$Z(G(C^g);s(h,i):h\in C^g,i\in\NN)$ be the 
complete cycle index of genus $g$. 

Let $T$ be a map defined as follows: 
for each $h=({\bf a_1},\ldots,{\bf a_n})\in C^g$ and $i\in \{1,\ldots,n\}$, 
if 
${\bf a_i}={\bf 0}$, then 
\[
s(h,1)\mapsto x_{{\bf a_i}}^{1/q^g};
\]
if 
${\bf a_i}\neq {\bf 0}$, then 
\[
s(h,p)\mapsto x_{{\bf a_i}}^{p/q^g}. 
\]
Then we have 
\[
w_C^{(g)}(x_a:a\in \FF_q^g)
=
T(Z(G(C^g);s(h,i):h\in C^g,i\in \NN)). 
\]
\end{thm}

\begin{proof}
Let $h=({\bf a_1},\ldots,{\bf a_n})\in C^g$ and 
\[
\wt^{(g)}(h)=\sharp\{i\mid {\bf a_i}\neq {\bf 0}\}. 
\]
If ${\bf a_i}= 0$, then the $q^g$ points 
of the form $(i,{\bf x})\in \{1,\ldots,n\}\times \FF_q^g$ 
are all fixed by this element; 
if ${\bf a_i}\neq 0$, they are
permuted in $q/p$ cycles of length $p$. 
Thus, $h=({\bf a_1},\ldots,{\bf a_n})\in C^g$ contributes 
\[
s(h,1)^{q^g(n-\wt^{(g)}(h))}s(h,p)^{q^g/p\wt^{(g)}(h)}
\]
to the sum in the formula for the complete cycle index, 
and 
\[
x_{\bf a_1}x_{\bf a_2}\cdots x_{\bf a_n}
\]
to the sum in the formula for the complete weight enumerator. 
The result follows. 
\end{proof}

To explain a relation between a 
higher cycle index and a 
higher weight enumerator, 
we review the concept behind it. 

\begin{df}
Let $C$ be a code over $\FF_q$ of length $n$. We have
\[
\|D\|=|{\rm supp}(D)|, 
\]
where 
$|{\rm supp}(D)|=\{i\mid \exists v\in D,v_i\neq 0\}$. In addition, 
\begin{align*}
\left\{
\begin{array}{l}
d_r=d_r(C)=\min\{\|D\|\mid D\leq C,\dim (D)=r\}, \\
A_i^r=A_i^r(C)=|\{D\leq C\mid \dim (D)=r,\|D\|=i\}|. 
\end{array}
\right. 
\end{align*}
Then the higher-weight enumerator is defined as follows: 
\begin{align*}
w_C^r(x,y):&=\sum_{D\leq C,\dim(D)=r}x^{n-\|D\|}y^{\|D\|}\\
&=\sum_{i=0}^{n}A_i^r(C)x^{n-i}y^i. 
\end{align*}
\end{df}

\begin{thm}[\cite{DGO,DGO2}]\label{thm:DGO}
Let $C$ be a code over $\FF_q$ of length $n$. Then 
\[
w_C^{(g)}(x_{{\bf 0}}=x,x_{{\bf a}}=y\ ({\bf a}\neq 0))
=\sum_{r=0}^{g}[g]_r w_C^r(x,y), 
\]
where 
\[
[g]_r=
\left\{
\begin{array}{ll}
1&{\rm if\ } r=0\\
(q^g-1)
(q^g-q)
\cdots
(q^g-q^{r-1})&{\rm otherwise.}
\end{array} 
\right.
\]
\end{thm}

The following theorem gives a relation 
between the higher cycle index and the higher weight enumerator.
 
\begin{thm}
Let $C$ be a code over $\FF_q$ of length $n$, 
where $q$ is a prime power of $p$.
Then 
\[
Z(G(C^g);s_i:i\in \NN)=\sum_{r=0}^{g}[g]_r w_C^r(s_1^{q^g},s_p^{q^g/p}). 
\]
\end{thm}

\begin{proof}
We claim that 
\[
Z(G(C^g);s_i:i\in \NN)
=w_C^{(g)}(x_{{\bf 0}}=s_1^{q^g},x_{{\bf a}}=s_p^{q^g/p}\ ({\bf a}\neq {\bf 0})). 
\]
Let $h=({\bf a_1},\ldots,{\bf a_n})\in C^g$ and 
\[
\wt^{(g)}(h)=\sharp\{i\mid {\bf a_i}\neq {\bf 0}\}. 
\]
If ${\bf a_i}= 0$, then the $q^g$ points 
of the form $(i,{\bf x})\in \{1,\ldots,n\}\times \FF_q^g$ 
are all fixed by this element; 
if ${\bf a_i}\neq 0$, they are
permuted in $q/p$ cycles of length $p$. 
Thus, $h=({\bf a_1},\ldots,{\bf a_n})\in C^g$ contributes 
\[
s_1^{q^g(n-\wt^{(g)}(h))}s_p^{q^g/p\wt^{(g)}(h)}
\]
to the sum in the formula for the complete cycle index, 
and 
\[
x_{\bf a_1}x_{\bf a_2}\cdots x_{\bf a_n}
\]
to the sum in the formula for the complete weight enumerator. 

The result follows by Theorem \ref{thm:DGO}. 
\end{proof}

\section{$\ZZ_{k}$-code analogue of the main results}
In \cite{BDHO}, 
the authors introduced the concept of 
$\ZZ_{k}$-codes. 
In this section, 
we give a $\ZZ_k$-code analogue of Theorem \ref{thm:main}. 

Let $\ZZ_{k}$ be the ring 
of integers modulo $k$, where $k$ 
is a positive integer. 
In this paper, we always assume that $k\geq 2$ and 
we take the set $\ZZ_{k}$ to be 
$\{0,1,\ldots,k-1\}$.
A $\ZZ_{k}$-code $C$ of length $n$
is a $\ZZ_{k}$-submodule of $\ZZ_{k}^n$.

The complete weight enumerator of genus $g$ is 
\[
w_C^{(g)}(x_{{\bf a}}:{\bf a}\in \ZZ_k^g)=\sum_{{\bf v_1},\ldots,{\bf v_g}\in C}
\prod_{{\bf a}\in \ZZ_k^g}x_{{\bf a}}^{n_{{\bf a}}({\bf v_1},\ldots,{\bf v_g})}, 
\]
where $n_{{\bf a}}({\bf v_1},\ldots,{\bf v_g})$ denotes the number 
of $i$ such that ${\bf a}=(v_{1i},\ldots,v_{gi})$. 

The following theorem is a $\ZZ_k$-code analogue of Theorem \ref{thm:main}. 

\begin{thm}\label{thm:ZZ}
Let $C$ be a code over $\ZZ_k$ of length $n$.
Let $w_C^{(g)}(x_{\bf a}:{\bf a}\in \ZZ_k^g)$ be the 
complete weight enumerator of genus $g$ and let 
$Z(G(C^g);s(h,i):h\in C^g,i\in\NN)$ be the 
complete cycle index of genus $g$. 

Let $T$ be a map defined as follows: 
for each $h=({\bf a_1},\ldots,{\bf a_n})\in C^g$ and $i\in \{1,\ldots,n\}$, 
\begin{align*}
&s(h,1)\mapsto x_{{\bf a_i}}^{1/k^g} for\ {\bf a_i}={\bf 0},\\
&s(h,k/\gcd(a_{i1},\ldots,a_{ig},k))\mapsto x_{{\bf a_i}}^{(k/\gcd(a_{i1},\ldots,a_{ig},k))/k^g}\ for\ {\bf a_i}={\bf 0}. 
\end{align*}
Then 
\[
w_C^{(g)}(x_a:a\in \ZZ_k^g)
=
T(Z(G(C^g);s(h,i):h\in C^g,i\in \NN)). 
\]
\end{thm}

\begin{proof}
Let $h=({\bf a_1},\ldots,{\bf a_n})\in C^g$ and 
\[
\wt^{(g)}(h)=\sharp\{i\mid {\bf a_i}\neq {\bf 0}\}. 
\]
If ${\bf a_i}= 0$, then the $k^g$ points 
of the form $(i,{\bf x})\in \{1,\ldots,n\}\times \ZZ_k^g$ 
are all fixed by this element; 
if ${\bf a_i}\neq 0$, they are
permuted in $(k/\gcd(a_{i1},\ldots,a_{ig},k))/k^g$ cycles of length $k/\gcd(a_{i1},\ldots,a_{ig},k)$. 
Then the result follows from the argument of Theorem \ref{thm:main}. 
\end{proof}

\section*{Acknowledgments}
The authors would also like to thank the anonymous
reviewers for their beneficial comments on an earlier version of the manuscript.





\end{document}